\documentclass{amsart}
\usepackage[english]{babel}
\usepackage[latin1]{inputenc}
 \usepackage[all]{xy}

\usepackage{amsmath,amsfonts,amssymb,amsthm,amscd,array,stmaryrd,mathrsfs, mathdots}
\usepackage{pstricks}

\theoremstyle{plain}
\newtheorem{thm}{Theorem}%[section]

\newtheorem{lem}{Lemma}[section]

\newtheorem{prop}[lem]{Proposition}
\newtheorem{conj}[lem]{Conjecture}

\theoremstyle{definition}
\newtheorem{defn}[lem]{Definition}
\newtheorem{rem}[lem]{Remark}
\newtheorem{ex}[lem]{Example}

\let\ssection=\section
\renewcommand{\section}{\setcounter{equation}{0}\ssection}

\newcommand{\R}{\mathbb{R}}
\newcommand{\Z}{\mathbb{Z}}
\newcommand{\C}{\mathbb{C}}

\newcommand{\Q}{\mathbb{Q}}

\newcommand{\A}{\mathcal{A}}
\newcommand{\B}{\mathcal{B}}
\newcommand{\cC}{\mathcal{C}}

\newcommand{\Pc}{\mathcal{P}} 
 
\newcommand{\Rc}{\mathcal{R}}
\newcommand{\Sc}{\mathcal{S}}

\newcommand{\Tr}{\textup{Tr\,}}

\newcommand{\Id}{\mathrm{Id}}
\newcommand{\GL}{\mathrm{GL}}
\newcommand{\PG}{\mathrm{PSL}_{q}(2,\Z)}
\newcommand{\PGL}{\mathrm{PGL}}
\newcommand{\SL}{\mathrm{SL}}
\newcommand{\PSL}{\mathrm{PSL}}

\newcommand{\U}{\mathrm{U}}

\def\c{\gamma}

\begin{document}

\title[$q$-deformations of the modular group and of quadratic irrationals]{$q$-deformations of the modular group and of the real quadratic irrational numbers}

\author[L. Leclere, S. Morier-Genoud]{Ludivine Leclere and Sophie Morier-Genoud}

\address{Sophie Morier-Genoud,
Sorbonne Universit\'e, Universit\'e de Paris, CNRS,
Institut de Math\'ematiques de Jussieu-Paris Rive Gauche,
 F-75005, Paris, France
}

\address{
Ludivine Leclere,
Sorbonne Universit\'e, Universit\'e de Paris, CNRS,
Institut de Math\'ematiques de Jussieu-Paris Rive Gauche,
 F-75005, Paris, 
 France}
\email{sophie.morier-genoud@imj-prg.fr, ludivine.leclere1@gmail.com}

\keywords{$q$-analogues, quadratic irrationals, continued fractions, modular group, palindromic polynomials, unimodality, continuants}

\maketitle

\begin{abstract}
We develop further the theory of $q$-deformations of real numbers introduced in \cite{MGOfmsigma} and \cite{MGOexp} and focus in particular on the class of real quadratic irrationals. Our key tool is  a $q$-deformation of the modular group $\PG$. The action of the modular group by M\"obius transformations commutes with the $q$-deformations. We prove that the traces of the elements of $\PG$ are palindromic polynomials with positive coefficients. These traces appear in the explicit expressions of the $q$-deformed quadratic irrationals. 
\end{abstract}

\thispagestyle{empty}
\tableofcontents
\section{Introduction}
A $q$-deformation of rational and irrational numbers was introduced in \cite{MGOfmsigma} and \cite{MGOexp}. The approach is based on combinatorial properties of the rational numbers. The subject can be linked to classical topics such as the Markov-Hurwitz approximation theory \cite{LMGOV}, see aslo \cite{GiVa}, \cite{Kog}, combinatorics of posets \cite{McCSS}, knots invariants \cite{KoWa}.

The $q$-deformation of a rational number is a rational function in the parameter $q$ with integer coefficients. It can be obtained using a $q$-deformation of the classical recursive procedure generating the rationals via Farey sums. Alternatively, it can be computed using explicit formulas involving a $q$-deformation of the continued fraction expansion of the number. 
It is sometimes convenient to consider the $q$-rational number as a formal power series using the Taylor expansion of the rational function.

 It was proved in \cite{MGOexp} that if $(x_{n})_{n}$ is an arbitrary sequence of rational numbers converging to an irrational number $x$, then the sequence of the power series of the $q$-rationals $([x_{n}]_{q})_{n}$ converges to a power series depending only on $x$ and this allows to define the limit as $[x]_{q}$ the $q$-deformation of $x$.

Many natural questions arise. Among them we focus on the following:

(1) How do algebraic operations on the real numbers behave under $q$-deformations?

(2) Do the $q$-deformations of algebraic numbers satisfy algebraic equations?

Regarding question (1) our main result is
\begin{thm}\label{relx} For all $x\in \R$, the $q$-deformations satisfy
\begin{eqnarray}
[x+1]_{q}&=&q[x]_{q}+1,\\\label{relx1}
\left[-\frac{1}{x}\right]_{q}&=&-q^{-1}\frac{1}{[x]_{q}}.\label{relx2}
\end{eqnarray}
\end{thm}

Regarding question (2) we restrict ourselves to the class of real quadratic irrationals and obtain the following results.
\begin{thm}\label{result1}
Let $x=\frac{r\pm\sqrt{p}}{s}$ be a quadratic irrational. Its $q$-deformation $[x]_{q}$ satisfy the following 
\begin{itemize}
\item[(i)] \label{genexp} 
$[x]_{q}=\frac{\Rc\pm\sqrt\Pc}{\Sc}$, with $\Rc, \Pc, \Sc \in \Z[q]$, and $ \Pc$ a palindrome;
\item[(ii)] 
$[x]_{q}$ is solution of an equation $\A X^{2}+\B X+\cC=0$, with $\A, \B, \cC \in \Z[q]$;
\item[(iii)]\label{matit} there exists a matrix $M_{q}\in\GL(2,\Z[q^{\pm1}])$ such that $M_{q}\cdot [x]_{q}=[x]_{q}$;
\item[(iv)] $[x]_{q}$ has a periodic continued fraction expansion.
\end{itemize}
\end{thm}

One of the main ingredient in our approach is a $q$-deformation of the modular group $\PSL(2,\Z)$. More precisely one considers the following matrices that are $q$-deformations of standard generators of $\PSL(2,\Z)$:
$$
R_{q}:=
\begin{pmatrix}
q&1\\[4pt]
0&1
\end{pmatrix},
\qquad
S_{q}:=
\begin{pmatrix}
0&-q^{-1}\\[4pt]
1&0
\end{pmatrix},
$$
and one defines $\PSL_{q}(2,\Z)$ as the subgroup generated by the classes of $R_{q}$ and $S_{q}$ in the quotient group
$\GL(2, \Z[q^{\pm 1}])/\{\pm q^{N}\Id, N\in \Z\}$.
\begin{prop}\label{iso}
The group $\PG$ is isomorphic to $\PSL(2,\Z)$.
\end{prop}
As a reformulation of Theorem \ref{relx} one obtains that the action of the modular group by M\"obius transformations commute with the $q$-deformations. This is a key tool for the proof of Theorem \ref{genexp}. 
Moreover we obtain the following interesting combinatorial property on the traces of the elements of $\PG$.

\begin{thm}\label{tracePG}
The traces of the elements in $\PG$ are palindromic polynomials in $\Z[q]$ with positive integer coefficients, modulo a multiplicative factor $\pm q^{N}$. \end{thm}

In all the examples we observe furthermore that the sequence of coefficients in the traces is unimodal, and we formulate this property as a conjecture.

Theorem \ref{tracePG} implies the palindromicity property of the polynomial $ \Pc$ in
Theorem \ref{genexp}.
Usually $q$-analogues of positive integers are given by polynomials in $q$ with positive integer coefficients. One could expect that the polynomial $\Pc$ has positive coefficients, but this is not the case in general. 
However we observe on all the examples we have computed that this polynomial always factors out with a non-trivial polynomial that has positive coefficients.

The paper is organized as follows.

In Section \ref{secdef} we give the definitions of the $q$-deformed rational and irrational numbers introduced in \cite{MGOfmsigma} and \cite{MGOexp}. We present a more general definition for the $q$-irrationals and obtain a new description using continued fraction expansions with arbitrary integer coefficients, Theorem \ref{genfrac}. We also describe $q$-irrational numbers using infinite continued fraction expansions.

In Section \ref{secgroup} we introduce a $q$-deformation of the group $\PSL(2,\Z)$. This $q$-deformation is the main tool to prove Theorem \ref{result1} and Theorem \ref{genfrac}. The entries of the $q$-deformed matrices are Laurent polynomials in $q$ with integer coefficients. We prove Theorem \ref{tracePG} and conjecture the unimodality property of the traces. We illustrate the results with examples of Cohn matrices arising in the theory of Markov triples.
%At the end of the section we compute examples of $q$-deformed Cohn matrices.

In Section \ref{secquad} we focus on the case of quadratic irrationals. The general expression for the $q$-quadratic irrationals given in Theorem \ref{result1} (i) was not \textit{a priori} guaranteed from the deformation process introduced in \cite{MGOexp} but it becomes clear from the view point of the matrix action. Explicit formulas for the polynomials $\Rc$, $\Pc$, $\Sc$ and  $\A$, $\B$, $\cC$ are given in terms of continuants, Propositions \ref{ABC} and \ref{PRS}.
Finally we give concrete examples of $q$-quadratic irrationals.

\section{$q$-deformed numbers and $q$-continued fractions}\label{secdef}
Following \cite{MGOfmsigma} and \cite{MGOexp}, we define the $q$-deformation $[x]_{q}$ of every real numbers $x$ as a Laurent series with integer coefficients. In the latter papers the definition of $[x]_{q}$ was given for $x>1$ and then extend to all reals using a recursion $[x-1]_{q}=q^{-1}[x]_{q}-q^{-1}$. Here we use a uniformized and equivalent definition for all reals $x$.
\subsection{$q$-rationals} We adopt 
the following classical $q$-deformations of integers
\begin{eqnarray*}
[n]_{q}&=&\frac{1-q^n}{1-q}=1+q+q^{2}+\cdots+q^{n-1}\\[6pt]
[-n]_{q}&=&\frac{1-q^{-n}}{1-q}=-q^{-1}-q^{-2}-\cdots-q^{-n}.
\end{eqnarray*}
for any $n\in \Z_{> 0}$.
Obviously one has $[n]_{q}\in \Z_{> 0}[q]$ and $[-n]_{q}\in \Z_{<0}[q^{-1}]$. We also assume $[0]_{q}=0$.

Considering a rational $\frac{r}{s}\in \Q$ we always assume $r$ and $s$ coprime and $s>0$. We use continued fraction expansions. Every $\frac{r}{s}\in \Q$ has a unique expression of the form
\begin{equation}\label{regfrac}
\frac{r}{s}
\quad=\quad
a_1 + \cfrac{1}{a_2 
          + \cfrac{1}{\ddots +\cfrac{1}{a_{2m}} } }
\end{equation}
with $a_{i}\in \Z$ and $a_{i}\geq 1$ for $i\geq2$. We use the notation $[a_1,a_2,\ldots, a_{2m}]$ for the RHS of \eqref{regfrac}.

Following \cite{MGOfmsigma} one defines the $q$-deformation of $\frac{r}{s}=[a_1,a_2,\ldots, a_{2m}]$ as

\begin{equation}\label{qregfrac}
\textstyle[\frac{r}{s}]_{q}:=
[a_1]_{q} + \cfrac{q^{a_{1}}}{[a_2]_{q^{-1}} 
          + \cfrac{q^{-a_{2}}}{[a_{3}]_{q} 
          +\cfrac{q^{a_{3}}}{[a_{4}]_{q^{-1}}
          + \cfrac{q^{-a_{4}}}{
        \cfrac{\ddots}{[a_{2m-1}]_q+\cfrac{q^{a_{2m-1}}}{[a_{2m}]_{q^{-1}}}}}
          } }} .
\end{equation}
The continued fraction in the RHS of equation \eqref{qregfrac} is denoted by
$[a_{1}, \ldots, a_{2m}]_{q}$.

Alternatively one can use the negative continued fraction expansion, also known as \textit{Hirzebruch-Jung continued fractions}. 
Every $\frac{r}{s}\in \Q$ has a unique expression of the form
\begin{equation}\label{negfrac}
\textstyle\frac{r}{s}
=c_1 - \cfrac{1}{c_2 
          - \cfrac{1}{\ddots - \cfrac{1}{c_{k}} } }
\end{equation}
with $c_{i}\in \Z$ and $c_{i}\geq 2$ for $i\geq2$. As in \cite{Hir} we use the notation $\llbracket{}c_1,c_2,\ldots, c_{k}\rrbracket{}$ for the RHS of \eqref{negfrac}.

One obtains the $q$-deformation of $\frac{r}{s}=\llbracket{}c_1,c_2,\ldots, c_{k}\rrbracket{}\in \Q$ as
\begin{equation}\label{qnegfrac}
\textstyle[\frac{r}{s}]_{q}
=[c_1]_{q} - \cfrac{q^{c_{1}-1}}{[c_2]_{q} 
          - \cfrac{q^{c_{2}-1}}{\ddots - \cfrac{q^{c_{k-1}-1}}{[c_{k}]_{q}} } }.
\end{equation}
The continued fraction in the RHS of equation \eqref{qnegfrac} is denoted by $\llbracket{}c_1,c_2,\ldots, c_{k}\rrbracket{}_{q}$. 

The fact that the two deformations \eqref{qregfrac} and \eqref{qnegfrac} coincide is not obvious at first sight and was proved in \cite{MGOfmsigma} when the coefficients satisfy $a_{i}>1$ and $c_{i}\geq 2$ for all $i$. It turns out that this result can be extended to fractions with arbitrary integer coefficients. The following theorem will be deduced from results of the next section, see \S\ref{genfracpf}.
\begin{thm}\label{genfrac}
Let $a_{1}, \ldots, a_{2m}$ and $ c_1,c_2,\ldots, c_{k}$ be sequences of integers such that the continued fractions 
 $[a_{1}, \ldots, a_{2m}]_{}$ and $\llbracket{}c_1,c_2,\ldots, c_{k}\rrbracket{}_{}$ are well defined and are equal to the same rational number $\frac{r}{s}$. One has
$$[a_{1}, \ldots, a_{2m}]_{q}=\llbracket{}c_1,c_2,\ldots, c_{k}\rrbracket{}_{q}=\left[\frac{r}{s}\right]_{q}.$$
\end{thm}

\begin{rem}
In \cite{MGOfmsigma} the $q$-deformations of rational numbers was defined in the case of numbers greater than 1, i.e. for $a_{1}\geq1$ and $c_{1}\geq 2 $ in the expansions \eqref{regfrac} and \eqref{negfrac}. The notion was later extend to all rational numbers in \cite{MGOexp} using the recursion
$$
[x-1]_{q}=q^{-1}[x]_{q}-q^{-1}.
$$
It turns out that this way to extend the definition is equivalent to allow $a_{1}\in \Z$ and $c_{1}\in \Z$ in the expansions \eqref{regfrac} and \eqref{negfrac},  see \cite{TheseLu} for details.
\end{rem}

\begin{ex} (a) Here are some examples that can be computed directly from the definitions
\begin{equation}
\begin{array}{cccccccccc}
\left[-\frac{5}{3}\right]_{q}&=&[-2,3]_{q}&=&\llbracket -1,2,2\rrbracket_{q}&=&-q^{-2}\,\frac{1+2q+q^{2}+q^{3}}{1+q+q^{2}}\\[6pt]
\left[-\frac{1}{4}\right]_{q}&=&[-1,1,3]_{q}&=&\llbracket 0,4\rrbracket_{q}&=&-q^{-1}\frac{1}{1+q+q^{2}+q^{3}}\\[6pt]
\left[\frac{5}{12}\right]_{q}&=&[0,2,2,2]_{q}&=&\llbracket 1,3,2\rrbracket_{q}&=&q^{2}\frac{1+2q+q^{2}+q^{3}}{1+2q+3q^{2}+3q^{3}+2q^{4}+q^{5}}\\[6pt]
\left[\frac{3}{5}\right]_{q}&=&[0,1,1,2]_{q}&=&\llbracket 1,2,4,2\rrbracket_{q}&=&q^{}\frac{1+q+q^{2}}{1+2q+q^{2}+q^{3}}\\[6pt]
\left[\frac{5}{3}\right]_{q}&=&[1,1,1,1]_{q}&=&\llbracket 2,3 \rrbracket_{q}&=&\frac{1+q+2q^{2}+q^{3}}{1+q+q^{2}}\\[6pt]
%\left[\frac{4}{9}\right]_{q}&=&[]_{q}&=&\llbracket \rrbracket_{q}&=&q^{-}\frac{1+}{1+}\\[6pt]
%\left[\frac{9}{4}\right]_{q}&=&[]_{q}&=&\llbracket \rrbracket_{q}&=&\frac{1}{1+}\\[6pt]
\left[\frac{12}{5}\right]_{q}&=&[2,2,1,1]_{q}&=&\llbracket 3,2,3\rrbracket_{q}&=&\frac{1+2q+3q^{2}+3q^{3}+2q^{4}+q^{5}}{1+q+2q^{2}+q^{3}}\\[6pt]
\end{array}
\end{equation}
(b) Only very few $q$-deformations of rationals are obtained as the ratio of the $q$-integers in the numerators and denominators. For instance for $r\in \Z_{>0}$,
$$
\left[\frac{r+1}{r}\right]_{q}=\frac{[r+1]_{q}}{[r]_{q}},\qquad \left[\frac{r}{r+1}\right]_{q}=q\frac{[r]_{q}}{[r+1]_{q}},
$$

$$
  \left[-\frac{r+1}{r}\right]_{q}=-q^{-2}\frac{[r+1]_{q}}{[r]_{q}},\qquad \left[-\frac{r}{r+1}\right]_{q}=-q^{-1}\frac{[r]_{q}}{[r+1]_{q}}.
$$
\end{ex}

\begin{ex}
Let us illustrate Theorem \ref{genfrac}. Alternatively, 
%$-\frac{1}{4}=[0,-4]$, 
$\frac53=[2,-1,-1,2]
=\llbracket -1,0,3,3 \rrbracket$ and one computes
%\begin{eqnarray*}
%[0,-4]_{q}&=&\frac{1}{[-4]_{q^{-1}}}=\frac{1}{-q-q^{2}-q^{3}-q^{4}}=-q^{-1}\frac{1}{1+q+q^{2}+q^{3}}\\
%\end{eqnarray*}
$$[2,-1,-1,2]_{q}=1+q + \cfrac{q^{2}}{-q
          + \cfrac{q^{}}{-q^{-1}
          +\cfrac{q^{-1}}{1+q^{-1}         
          } }}=\frac{1+q+2q^{2}+q^{3}}{1+q+q^{2}},$$
          
 $$\llbracket-1,0,3,3\rrbracket_{q}=-q^{-1} - \cfrac{q^{-2}}{0
          - \cfrac{q^{-1}}{1+q+q^{2}
          -\cfrac{q^{2}}{1+q+q^{2}        
          } }}=\frac{1+q+2q^{2}+q^{3}}{1+q+q^{2}}.$$
\end{ex}

The $q$-deformations $[a_{1}, \ldots, a_{2m}]_{q}$ and $\llbracket{}c_1,c_2,\ldots, c_{k}\rrbracket{}_{q}$ can be written as rational functions in $q$ with integer coefficients. 
The general expression for $[\frac{r}{s}]_{q}$ is as follows.
\begin{prop}[\cite{MGOfmsigma}]\label{R/S}
Let $\frac{r}{s}$ be a non zero rational number.

(i) There exist a unique couple of coprime polynomials $\Rc$ and $\Sc$ in $ \Z_{}[q]$ and $N\in \Z$ such that 
$$
\left[\frac{r}{s}\right]_{q}=
\pm q^{-N}\frac{\Rc}{\Sc},
$$
where the sign coincides with the sign of $\frac{r}{s}$;

(ii)
The polynomials $\Rc$ and $\Sc$ have positive integer coefficients, they satisfy $\Rc(1)=|r|$, $\Sc(1)=s$, and have constant and leading coefficients equal to 1;

(iii) If $\frac{r}{s}\geq1$ then $N=0$, otherwise $N$ is characterized  by
$$
-N\leq \frac{r}{s}<-N+1, \quad \text{ or } \quad 0<\frac{1}{1-N}<\frac{r}{s}\leq \frac{1}{-N}< 1.
$$
\end{prop}

\begin{rem}
In \cite{MGOfmsigma}, it was conjectured that $\Rc$ and $\Sc$ have unimodal sequences of coefficients. This was recently proved in some particular cases \cite{McCSS} but the conjecture is still open in full generality.
\end{rem}

\subsection{$q$-irrationals}
The $q$-rationals defined in the previous section can be viewed as Laurent series with integer coefficients, i.e. 
$$
\left[\frac{r}{s}\right]_{q}=\sum \limits_{k=-N}^{+\infty} \rho_{k} q^k
$$
with $N\in\Z_{\geq 0}$  as in Proposition \ref{R/S}, $\rho_{-N}=\pm 1$, and $\rho_{k}\in \Z$. If one is given a convergent sequence of rationals then the coefficients in the corresponding power series stabilize. This was proved in \cite{MGOexp} in the case of numbers greater than one. This phenomenon still hold in full generality, see \cite{TheseLu}.

\begin{thm}[\cite{MGOexp}]\label{stab}
Let $x\in \R$ be an irrational number and ${(x_{n})}_{n}$ a sequence of rationals converging to $x$, and let $\left[x_{n}\right]_{q}=\sum \limits_{k}^{} \chi_{n,k} q^k$ be the corresponding $q$-deformations expanded as Laurent series.

(i)
 For all~$k$, the sequence $( \chi_{n,k})_{n}$ converges;
 
 (ii) 
 The limit coefficient $\chi_{k}:=\lim_{n\to \infty}\chi_{n,k}$ is
an integer which depends only on $x$ and not on the choice of the converging sequence $(x_{n})_{n}$;

(iii) There exists $N\in \Z_{\geq 0}$ such that for all $k<-N$ one has $\chi_{k}=0$ and $\chi_{(-N)}=\pm1$.
\end{thm}
This allows the following definition.
\begin{defn}\cite{MGOexp}
The $q$-deformation of an irrational number $x$
is the Laurent series
$$
[x]_{q}:=  \sum \limits_{k=-N}^{+\infty} \chi_{k} q^k.
$$
where $N\in \Z_{\geq 0}$ and  $\chi_{k}\in \Z$ are given by Theorem \ref{stab}.
\end{defn}
%Note that since we are considering sequences of integers $(\chi_{n,k})$ this means that there exists an integer $N_{k}$ such that $\chi_{k}=\chi_{n,k}$ for all $n\geq N_{k}$. In other words, the coefficients in the series $\left[\frac{r_{n}}{s_{n}}\right]_{q}$ and $[x]_{q}$ coincide up to the order $N_{k}$.

\begin{ex}
Theorem \ref{stab} says that if two rational numbers are close to each other (i.e. the difference is close to 0) then their $q$-deformations are close to each other (i.e. the Taylor series expansions coincide up to some order). For instance if one considers the rationals $\frac{12}{5}=2.4$, $\frac{241}{100}=2.41$ and $\frac{408}{169}\approx2.41420$ their $q$-deformations are quite different as rational fractions but their Taylor expansions coincide up to order 7.
One has 
\begin{eqnarray*}
\left[\frac{12}{5}\right]_{q}&=&
\frac{1+2q+3q^{2}+3q^{3}+2q^{4}+q^{5}}{1+q+2q^{2}+q^{3}}\\[1pt]
&=&1+q+q^{4}-2q^{6}+q^{7}+3q^{8}-3q^{9}-4q^{10}+7q^{11}+4q^{12}+\ldots \\[10pt]
\left[\frac{241}{100}\right]_{q}&=&\frac{q^{12} + 4q^{11} + 10q^{10} + 20q^9 + 29q^8 + 37q^7 + 40q^6 + 37q^5 + 29q^4 + 19q^3 + 10q^2 + 4q + 1}{q^{10} + 4q^9 + 8q^8 + 13q^7 + 17q^6 + 18q^5 + 16q^4 + 12q^3 + 7q^2 + 3q + 1}\\[1pt]
&=&1+q+q^{4}-2q^{6}+q^{7}+3q^{8}-2q^{9}-7q^{10}+9q^{11}+7q^{12}-17q^{13}+\cdots\\[10pt]
  \left[ \frac{408}{169}\right]_{q}&= &\frac{1 + 4q+ 12q^{2} + 25q^{3} + 41q^4 + 56q^5 + 65q^6 + 65q^{7}+56q^{8}+41q^{9}+25q^{10}+12q^{11}+ 4q^{12}  + q^{13}}{1 + 3q+ 9q^{2} + 16q^{3} + 24q^4 + 29q^5 + 29q^6 + 24q^{7}+16q^{8}+9q^{9} + 4q^{10}  + q^{11}} \\[5pt]
&=&1+q+q^{4}-2q^{6}+q^{7}+4q^{8}-5q^{9}-7q^{10}+18q^{11}+7q^{12}-55q^{13}+18q^{14}+146q^{15}-156q^{16}\ldots
%-316q^{17}+\ldots
  \end{eqnarray*}
These numbers are close to $1+\sqrt2\approx2.41421$. It turns out that
$\left[1+\sqrt2\right]_{q}  $ can be computed explicitly (see Section \ref{exquad}) and one obtains
\begin{eqnarray*}
[1+\sqrt{2}]_{q}& = &\cfrac {q^{3}+2q-1+ \sqrt{q^{6}+4q^{4}-2q^{3}+4q^{2}+1}}{2q}\\
&=&1+q+q^{4}-2q^{6}+q^{7}+4q^{8}-5q^{9}-7q^{10}+18q^{11}+7q^{12}-55q^{13}+18q^{14}+146q^{15}-155q^{15}\ldots
%-322q^{17}+\ldots
\end{eqnarray*}
which coincide up to order 16 with $ \left[ \frac{408}{169}\right]_{q}$.
%R/S = (q^12 + 4q^11 + 10q^10 + 20q^9 + 29q^8 + 37q^7 + 40q^6 + 37q^5 + 29q^4 + 19q^3 + 10q^2 + 4q + 1)/ (q^10 + 4q^9 + 8q^8 + 13q^7 + 17q^6 + 18q^5 + 16q^4 + 12q^3 + 7q^2 + 3q + 1) 
% 7*q^12 + 9*q^11 - 7*q^10 - 2*q^9 + 3*q^8 + q^7 - 2*q^6 + q^4 + q + 1
\end{ex}

\subsection{Proof of Theorem \ref{relx}}
From the explicit formula \eqref{qregfrac} 
%and \eqref{qnegfrac} 
%and Theorem \ref{genfrac} 
one can easily check the relations of Theorem~\ref{relx} in the case where $x$ is a rational. More precisely one gets

\begin{prop}
For $x\in \Q$, $n\in \Z$, one has
\begin{eqnarray}\label{relrec}
[x+n]_{q}&=&q^{n}[x]_{q}+[n]_{q},\\[6pt]
\label{relneg}
[-x]_{q}&=&-q^{-1}[x]_{q^{-1}},\\[6pt]
\label{relinv}
\left[\frac{1}{x}\right]_{q}&=&\frac{1}{[x]_{q^{-1}}}.
\end{eqnarray}

\end{prop}

 The relations \eqref{relneg} and \eqref{relinv} do not make sense anymore for real numbers, as they involved comparison of series in $q$ and in $q^{-1}$, however their combination in the form \eqref{relx2} does make sense. The relations \eqref{relx1} and \eqref{relx2} holding in the case of $q$-rationals will be preserved in the limit process defining the $q$-reals.

\subsection{Infinite continued fractions}\label{infCF}
It is well known that every irrational number can be written with infinite continued fraction expansions. For every irrational number $x$ there exist 
a sequence of integers ${(a_i)}_{i \ge 1}$ with $a_i \ge 1$ for every $i \ge 2$ such that the sequence of rationals $\llbracket{}a_1, \ldots, a_n\rrbracket$ converges to $x$. One writes
$$
x=
a_1 + \cfrac{1}{a_2
          + \cfrac{1}{a_{3}
          +\cfrac{1}{a_{4}
          + \cfrac{1}{
        {\ddots}}
          } }} .
$$
and uses the notation $x=[a_{1}, a_2, a_3 \ldots]$. 
By Theorem \ref{stab}, the sequence of rational fractions $\llbracket{}a_1, \ldots, a_{2m}\rrbracket_{q}$ converges to the formal power series $[x]_{q}$. One writes
$$
[a_{1}, a_2, a_3 \ldots]_{q}:=
[a_1]_{q} + \cfrac{q^{a_{1}}}{[a_2]_{q^{-1}} 
          + \cfrac{q^{-a_{2}}}{[a_{3}]_{q} 
          +\cfrac{q^{a_{3}}}{[a_{4}]_{q^{-1}}
          + \cfrac{q^{-a_{4}}}{
        {\ddots}}
          } }} .
$$
and uses the notation $[x]_{q}=[a_{1}, a_2, a_3 \ldots]_{q}$. 

Similarly with negative continued fractions, for every irrational number $x$ there exist 
a sequence of integers ${(c_i)}_{i \ge 1}$ with $c_i \ge 2$ for all $i \ge 2$ such that the sequence of rationals $\llbracket{}c_1, \ldots, c_n\rrbracket$ converges to $x$. One writes
\begin{equation*}
x
=c_1 - \cfrac{1}{c_2 
          - \cfrac{1}{c_{3} -
           \cfrac{1}{\ddots} }}
\end{equation*}
and uses the notation $x=\llbracket{}c_1, c_{2}, \ldots\rrbracket$. 

By Theorem \ref{stab}, the sequence of rational fractions $\llbracket{}c_1, \ldots, c_n\rrbracket_{q}$ converges to the formal power series $[x]_{q}$. One writes
$$
[x]_{q}=
[c_1]_{q} - \cfrac{q^{c_{1}-1}}{[c_2]_{q} 
          - \cfrac{q^{c_{2}-1}}{
          [c_{3}]_{q}- \cfrac{q^{c_{3}-1}}{
          \ddots }}} 
$$
and uses also the notation
$[x]_{q}=\llbracket{}c_1,c_2,c_3 \ldots\rrbracket_{q}$.

%%%%%%%%%%%%%%%%%%%%%      
      \section{$q$-deformations of matrices and of the modular group}\label{secgroup}
%%%%%%%%%%%%%%%%%%%%%%      

In this section we define the $q$-deformation of a matrix $M\in \SL(2,\Z)$ as an element
in $\PGL(2, \Z[q^{\pm 1}])$.
\subsection{Elementary matrices in $\SL(2,\Z)$}
First we introduce elementary matrices in $\SL(2,\Z)$ and recall some standard decomposition.
The elementary matrices
$$
R=\left(
\begin{array}{cc}
1&1\\[4pt]
0&1
\end{array}
\right),
\qquad
S=\left(
\begin{array}{cc}
0&-1\\[4pt]
1&0
\end{array}
\right),
$$
are one of the standard choice of generators of the group~$\SL(2,\Z)$. They satisfy the following relations
$$
(RS)^{3}=-\Id, \qquad S^{2}=-\Id
$$
We will always consider the matrices up to a sign, i.e. as elements of $\PSL(2,\Z)=\SL(2,\Z)/\{\pm\Id\}$. We still denote by $R$ and $S$ the images of $R$ and $S$ in  $\PSL(2,\Z)$.
 It will be convenient to also consider the elementary matrix
$$
L:=\left(
\begin{array}{cc}
1&0\\[4pt]
1&1
\end{array}
\right).
$$
The couple $(R,L)$ is also a standard choice of generators of $\SL(2,\Z)$. 

\subsection{$q$-deformation of the modular group (proof of Proposition \ref{iso})}
We consider the ring of Laurent polynomials with integer coefficients $\Z[q^{\pm 1}]$ and its group of units $\U:=\Z[q^{\pm 1}]^{\times}=\{\pm q^{N}, N\in \Z\}$. We consider the following groups of $2\times 2$-matrices
\begin{eqnarray*}
\GL(2, \Z[q^{\pm 1}])&=&\left\{
\begin{pmatrix} 
A&B\\
C&D
\end{pmatrix} \left|  A, B, C, D \in \Z[q^{\pm 1}] \right.: AD-BC\in \U\right\},\\[8pt]
\PGL(2, \Z[q^{\pm 1}])&=&
\GL(2, \Z[q^{\pm 1}])/ \U.
\end{eqnarray*}
As in \cite{MGOfmsigma} one introduces the following matrices
$$
R_{q}:=
\begin{pmatrix}
q&1\\[4pt]
0&1
\end{pmatrix},
\qquad
S_{q}:=
\begin{pmatrix}
0&-q^{-1}\\[4pt]
1&0
\end{pmatrix},
$$
that are $q$-analogues of $R$ and $S$
and which satisfy $
(R_{q}S_{q})^{3}=-\Id, \; S_{q}^{2}=-q^{-1}\Id
$. 

One
considers the subgroup generated by the classes of $R_{q}$ and $S_{q}$ inside $ \PGL(2, \Z[q^{\pm 1}])$. One defines
$$
\PG:=\langle R_{q}, S_{q}\rangle \subset \PGL(2, \Z[q^{\pm 1}]).
$$
In $\PG$ one has the relations $
(R_{q}S_{q})^{3}= S_{q}^{2}=\Id
$, so that the following assignment 
\begin{equation}\label{qmat}
\begin{array}{cccc}
[\,\cdot\,]_{q}:& R & \mapsto  & [R]_{q}=R_{q}  \\
  & S & \mapsto  & [S]_{q}=S_{q}  \\ 
\end{array}
\end{equation}
realizes an isomorphism between  $\PSL(2,\Z)$ and $\PG$.
Proposition \ref{iso} is proved.

One may consider the $q$-deformation $[M]_{q}$ of any matrices $M\in \PSL(2,\Z)$ via the map \eqref{qmat}.

\begin{ex}
Computing the decomposition $L=RSR$ one obtains the $q$-deformed matrix
$$[L]_{q}=\left(
\begin{array}{cc}
q&0\\[4pt]
q&1
\end{array}
\right).$$
\end{ex}
\subsection{M\"obius transformation of the $q$-reals}
The modular group $\PSL(2,\Z)$ acts on the projective real line via M\"obius transformations:
$$
M\cdot x :=\frac{ax+b}{cx+d}, \quad \forall x\in \R\cup\{\infty\}, \quad M=
\begin{pmatrix}
a&b\\c&d
\end{pmatrix}\in \PSL(2,\Z).
$$
We consider the $q$-analogue action of $\PSL(2,\Z)\simeq\PG$ on $\Z[[q]]\cup\{\infty\}$:
$$
M\cdot f :=\frac{Af+B}{Cf+D}, \quad \forall f\in \Z[[q]]\cup\{\infty\}, \quad M=
\begin{pmatrix}
A&B\\C&D
\end{pmatrix}\in \PG.
$$
\begin{prop}\label{action}
The $\PSL(2,\Z)$-actions commute with the $q$-deformations.
\end{prop}

\begin{proof}
We want to show that $[M\cdot x]_{q}=[M]_{q}\cdot [x]_{q}$. It suffices to consider the cases $M=R$ and $M=S$, for which we know that the equalities hold by Theorem \ref{relx}.
\end{proof}

\subsection{$q$-deformed matrices and continued fractions}
In this section we study the $q$-deformations of elementary matrices related to continued fractions. 
Let us consider the generalized continued fraction
$$
F_{n}=\quad
x_1 + \cfrac{y_{1}}{x_2 
          + \cfrac{y_{2}}{\ddots +\cfrac{y_{n-1}}{x_{n}} } },
$$
where $x_{i}$'s and $y_{i}$'s are viewd as formal variables.
It is well known, e.g. \cite{Fra}, that it can be obtained by $2\times 2$-matrix computations as below.

\begin{lem}\label{lemgen}
The matrix product
$$
\left(
\begin{array}{cc}
x_1&y_{1}\\[4pt]
1&0
\end{array}
\right)
\left(
\begin{array}{cc}
x_{2}&y_{2}\\[4pt]
1&0
\end{array}
\right)
\cdots
\left(
\begin{array}{cc}
x_{n}&y_{n}\\[4pt]
1&0
\end{array}
\right)=
\left(
\begin{array}{cc}
U_{n}&y_{n}U_{n-1}\\[4pt]
V_{n}&y_{n}V_{n-1}
\end{array}
\right)
$$
gives 
$$
\frac{U_{n}}{V_{n}}=F_{n} .
$$
\end{lem}
We introduce the matrices corresponding to the continued fraction expansions of types \eqref{regfrac} and \eqref{negfrac}:
\begin{equation}\label{matMM+}
\begin{array}{rcl}
M^+(a_1,\ldots,a_{2m}) &:=&
\left(
\begin{array}{cc}
a_1&1\\[4pt]
1&0
\end{array}
\right)
\left(
\begin{array}{cc}
a_2&1\\[4pt]
1&0
\end{array}
\right)
\cdots
\left(
\begin{array}{cc}
a_{2m}&1\\[4pt]
1&0
\end{array}
\right),
\\[16pt]
M(c_1,\ldots,c_k) &:=&
\left(
\begin{array}{cc}
c_1&-1\\[4pt]
1&0
\end{array}
\right)
\left(
\begin{array}{cc}
c_{2}&-1\\[4pt]
1&0
\end{array}
\right)
\cdots
\left(
\begin{array}{cc}
c_k&-1\\[4pt]
1&0
\end{array}
\right).
\end{array}
\end{equation}
defined for arbitrary sequences of integers $a_{1}, \ldots, a_{2m}$ and $ c_1,c_2,\ldots, c_{k}$.

We want to apply the map \eqref{qmat} to these matrices. To do so we write the matrices as product of the generators $R$ and $S$ or $R$ and $L$.
By direct computations one checks the following lemma.
\begin{lem}
The matrix decompositions in terms of the generators are
$$
\begin{array}{rcl}
M^+(a_1,\ldots,a_{2m})&=&
R^{a_1}L^{a_{2}}
R^{a_{3}}L^{a_{4}}\cdots{}
R^{a_{2m-1}}L^{a_{2m}},
\\[4pt]
M(c_1,\ldots,c_k)&=&R^{c_1}S\,R^{c_{2}}S\cdots{}R^{c_k}S.
\end{array}
$$
\end{lem}

Applying the map \eqref{qmat} one obtains the expressions for the $q$-deformed matrices.

\begin{lem} One has

(i)
$[M(c_{1},\ldots, c_{k})]_{q}=
\begin{pmatrix}
[c_{1}]_{q}&-q^{c_{1}-1}\\[6pt]
1&0
\end{pmatrix}
\begin{pmatrix}
[c_{2}]_{q}&-q^{c_{2}-1}\\[6pt]
1&0
\end{pmatrix}
\cdots
\begin{pmatrix}
[c_{k}]_{q}&-q^{c_{k}-1}\\[10pt]
1&0
\end{pmatrix};
$

(ii) $[M^{+}(a_{1},\ldots, a_{2m})]_{q}$\\

$
\begin{array}{rcc}
&=&
\begin{pmatrix}
q^{a_{1}}&[a_{1}]_{q}&\\[6pt]
0&1
\end{pmatrix}
\begin{pmatrix}
q^{a_{2}}&0\\[6pt]
q[a_{2}]_{q}& 1
\end{pmatrix}
\cdots
\begin{pmatrix}
q^{a_{2m-1}}&[a_{2m-1}]_{q}&\\[6pt]
0&1
\end{pmatrix}
\begin{pmatrix}
q^{a_{2m}}&0\\[6pt]
q[a_{2m}]_{q}& 1
\end{pmatrix},
\\[20pt]
&=&
  \begin{pmatrix}
[a_{1}]_{q}&q^{a_{1}}\\[6pt]
1&0
\end{pmatrix}
\begin{pmatrix}
q[a_{2}]_{q}& 1\\[6pt]
q^{a_{2}}&0
\end{pmatrix}
\cdots
\begin{pmatrix}
[a_{2m-1}]_{q}&q^{a_{2m-1}}\\[6pt]
1&0
\end{pmatrix}
\begin{pmatrix}
q[a_{2m}]_{q}& 1\\[6pt]
q^{a_{2m}}&0
\end{pmatrix},\\[20pt]
&=&q^{a_{2}+a_{4}+\ldots+a_{2m}}
\begin{pmatrix}
[a_{1}]_{q}&q^{a_{1}}\\[6pt]
1&0
\end{pmatrix}
\begin{pmatrix}
[a_{2}]_{q^{-1}}& q^{-a_{2}}\\[6pt]
1&0
\end{pmatrix}
\cdots
\begin{pmatrix}
[a_{2m-1}]_{q}&q^{a_{2m-1}}\\[6pt]
1&0
\end{pmatrix}
\begin{pmatrix}
[a_{2m}]_{q^{-1}}&q^{-a_{2m}}\\[6pt]
1&0
\end{pmatrix}.
\\[20pt]
\end{array}
$
\end{lem}

The matrices in the above lemma are elements in $\PG$. It will be convenient to fix representatives in $\GL(2, \Z[q^{\pm 1}])$. We define
\begin{equation}
\label{qNegMat}
M_{q}(c_{1},\ldots, c_{k}):=
\begin{pmatrix}
[c_{1}]_{q}&-q^{c_{1}-1}\\[6pt]
1&0
\end{pmatrix}
\begin{pmatrix}
[c_{2}]_{q}&-q^{c_{2}-1}\\[6pt]
1&0
\end{pmatrix}
\cdots
\begin{pmatrix}
[c_{k}]_{q}&-q^{c_{k}-1}\\[6pt]
1&0
\end{pmatrix}
\end{equation}
as a $q$-analogue of the matrix $M(c_{1},\ldots, c_{k})$ in $\GL(2, \Z[q^{\pm 1}])$, and
\begin{equation}
\label{qRegMat}
M^{+}_{q}(a_{1},\ldots, a_{2m}):=
\begin{pmatrix}
[a_{1}]_{q}&q^{a_{1}}\\[6pt]
1&0
\end{pmatrix}
\begin{pmatrix}
[a_{2}]_{q^{-1}}& q^{-a_{2}}\\[6pt]
1&0
\end{pmatrix}
\cdots
\begin{pmatrix}
[a_{2m-1}]_{q}&q^{a_{2m-1}}\\[6pt]
1&0
\end{pmatrix}
\begin{pmatrix}
[a_{2m}]_{q^{-1}}&q^{-a_{2m}}\\[6pt]
1&0
\end{pmatrix}
\end{equation}
as a $q$-analogue of the matrix  $M^{+}(a_{1},a_{2}, \ldots, a_{2m})$  in $\GL(2, \Z[q^{\pm 1}])$.

Using Lemma \ref{lemgen} we immediately see that the above matrices 
$M_{q}(c_{1},\ldots, c_{k})$ and $M^{+}_{q}(a_{1},\ldots, a_{2m})$ correspond to the 
$q$-deformed continued fractions of the types \eqref{qregfrac} and \eqref{qnegfrac}.
More precisely, one obtains:

\begin{lem}\label{keylem}
(i) Let $a_{1}, \ldots, a_{2m}$ be a sequence of integers such that $[a_{1}, \ldots, a_{2m}]_{q}$ given by \eqref{qregfrac} is a well defined rational function in $q$. One has
$$
[a_{1}, \ldots, a_{2m}]_{q}=\frac{\Rc^{+}}{\Sc^{+}},
$$
where ${\Rc^{+}}$ and ${\Sc^{+}}$ are polynomials in $\Z[q^{\pm 1}]$ given by the matrix
$$
M^{+}_{q}(a_{1},\ldots, a_{2m})=\begin{pmatrix}
{\Rc^{+}}&*\\[6pt]
\Sc^{+}&*
\end{pmatrix}.
$$

(ii) Let $c_{1}, \ldots, c_{k}$ be a sequence of integers such that $\llbracket{}c_1,c_2,\ldots, c_{k}\rrbracket{}_{q}$ given by \eqref{qregfrac} is a well defined rational function in $q$. One has
$$
\llbracket{}c_1,c_2,\ldots, c_{k}\rrbracket{}_{q}=\frac{\Rc^{}}{\Sc^{}},
$$
where ${\Rc^{}}$ and ${\Sc^{}}$ are polynomials in $\Z[q^{\pm 1}]$ given by the matrix
$$
M_{q}(c_1,c_2,\ldots, c_{k})=\begin{pmatrix}
{\Rc^{}}&*\\[6pt]
\Sc^{}&*
\end{pmatrix}.
$$

\end{lem}

\begin{rem}
Lemma \ref{keylem} was proved in \cite[Prop 4.3]{MGOfmsigma} for a particular choice of coefficients $a_{i}$'sand $c_{i}$'s.  The more general result of Lemma \ref{keylem} is the key step for the proof of 
Theorem \ref{genfrac}.
\end{rem}

\subsection{Proof of Theorem \ref{genfrac}} \label{genfracpf}
Let $\frac{r}{s}$ be a rational defined from sequences of integers $a_{1}, \ldots, a_{2m}$ and $c_{1}, \ldots, c_{k}$ by the continued fractions \eqref{regfrac} and \eqref{negfrac} respectively, i.e. $\frac{r}{s}=[a_{1}, \ldots, a_{2m}]=\llbracket{}c_1,c_2,\ldots, c_{k}\rrbracket{}$.
We consider the $q$-deformed continued fractions \eqref{qregfrac} and \eqref{qnegfrac} and write
$$
[a_{1}, \ldots, a_{2m}]_{q}=\frac{\Rc^{+}}{\Sc^{+}},\qquad 
\llbracket{}c_1,c_2,\ldots, c_{k}\rrbracket{}_{q}=\frac{\Rc^{}}{\Sc^{}}.
$$
We want to show $
[a_{1}, \ldots, a_{2m}]_{q}=
\llbracket{}c_1,c_2,\ldots, c_{k}\rrbracket{}_{q}$.

By Lemma \ref{lemgen} we know that the ratio of the first columns of the matrices $M^{+}(a_{1},\ldots, a_{2m})$ and $M(c_1,c_2,\ldots, c_{k})$ defined in \eqref{matMM+} gives the rational $\frac{r}{s}$.
Since the matrices belong to $\SL(2, \Z)$ and the rational $\frac{r}{s}$ is written in the irreducible form, we can write 
$$
M^{+}(a_{1},\ldots, a_{2m})=\begin{pmatrix}
r&u\\[4pt]
s&v
\end{pmatrix}, \quad M(c_1,c_2,\ldots, c_{k})=\begin{pmatrix}
r&u'\\[4pt]
s&v'
\end{pmatrix}.
$$
Since $rv-su=rv'-su'=1$ one has $u'=r+nu$ and $v'=s+nv$ for some $n\in\Z$. This implies the following relation on the matrices
$$
 M(c_1,c_2,\ldots, c_{k})=M^{+}(a_{1},\ldots, a_{2m})R^{n}.
$$
Applying the $q$-deformation of matrices \eqref{qmat} in the above identity and using Lemma \ref{keylem} one gets
\[
\begin{array}{rcl}
[ M(c_1,c_2,\ldots, c_{k})]_{q}&=&[M^{+}(a_{1},\ldots, a_{2m})]_{q}[R^{n}]_{q}\\[10pt]
\begin{pmatrix}
{\Rc^{}}&*\\[6pt]
\Sc^{}&*
\end{pmatrix}  &=   & \begin{pmatrix}
{\Rc^{+}}&*\\[6pt]
\Sc^{+}&*
\end{pmatrix}  \begin{pmatrix}
q^{n}&[n]_{q}&\\[6pt]
0&1
\end{pmatrix}\\[20pt]
\begin{pmatrix}
{\Rc^{}}&*\\[6pt]
\Sc^{}&*
\end{pmatrix}  &=   & q^{n}\begin{pmatrix}
{\Rc^{+}}&*\\[6pt]
\Sc^{+}&*
\end{pmatrix} \\[20pt]
\end{array}
\]
One deduces $\llbracket{}c_1,c_2,\ldots, c_{k}\rrbracket{}_{q}=\frac{\Rc^{}}{\Sc^{}}=\frac{\Rc^{+}}{\Sc^{+}}=[a_{1}, \ldots, a_{2m}]_{q}$. Theorem \ref{genfrac} is proved.

\subsection{$q$-Continuants} 
Entries in the matrices $M^+(a_1,\ldots,a_{2m}) $ and $M(c_1,c_2,\ldots, c_{k})$, or equivalently numerators and denominators of the corresponding continued fractions, may be explicitly given with the help of determinantal expressions known as Euler continuants. The $q$-analogues of Euler continuants have been introduced in \cite{MGOfmsigma}. We will mainly consider those related to negative continued fraction. Let us recall the definition and some properties
\begin{equation}
\label{KEq}
E_{k}(c_1,\ldots,c_k)_{q}:=
\left|
\begin{array}{cccccccc}
[c_1]_{q}&q^{c_{1}-1}&&&\\[6pt]
1&[c_{2}]_{q}&q^{c_{2}-1}&&\\[4pt]
&\ddots&\ddots&\!\!\ddots&\\[4pt]
&&1&\!\!\![c_{k-1}]_{q}&q^{c_{k-1}-1}\\[6pt]
&&&\!\!\!\!\!\!1&\!\!\!\!\!\!\!\![c_{k}]_{q}
\end{array}
\right|,
\end{equation}
where $c_{i}$'s are integers, and with convention $E_{0}()=1$ and $E_{-1}()=0$. 

When the $c_{i}$'s are positive integers $E_{k}(c_1,\ldots,c_k)_{q}$ is a polynomial in $q$ with positive integer coefficients. One has 
$$
\deg E_{k}(c_1,\ldots,c_k)_{q}=\sum_{i=1}^{k}c_{i}-k
$$
For any sequence of integers $(c_{1},\ldots, c_{k})$ one has
\begin{equation} \label{q-cont}
M_{q}(c_{1},\ldots, c_{k})=
\begin{pmatrix}
E_{k}(c_{1},\ldots, c_{k})&-q^{c_{k}-1}E_{k-1}(c_{1},\ldots, c_{k-1})\\[6pt]
E_{k-1}(c_{2},\ldots, c_{k})&-q^{c_{k}-1}E_{k-2}(c_{2},\ldots, c_{k-1})
\end{pmatrix}
\end{equation}
and 
$$
\llbracket{}c_1,c_2,\ldots, c_{k}\rrbracket{}_{q}=\frac{E_{k}(c_{1}, c_{2},\ldots, c_{k})_{q}}{E_{k-1}(c_{2}, c_{3},\ldots, c_{k})_{q}}.
$$

\subsection{Traces of $q$-deformed matrices (proof of Theorem \ref{tracePG})}
The sequence of letters $(\alpha_{1},\alpha_{2}, \ldots, \alpha_{n-1}, \alpha_{n})$ is a palindrome if it reads the same backward as forward, i.e. $\alpha_{1}=\alpha_{n}$, $\alpha_{2}= \alpha_{n-1}$, etc.
 A polynomial $\Pc\in \Z[q,q^{-1}]$ is said to be a palindrome, or to have palindromic coefficients, if its sequence of ordered coefficients is a palindrome. This means that  $\Pc$ is a palindrome if and only if $\Pc(q)=q^{N}\Pc(q^{-1})$ for some integer $N\in \Z$.

The traces of the $q$-deformed matrices are palindromes.

The goal of this section is to outline the proof of Theorem \ref{tracePG}.
The proof is based on three intermediate results that we formulate as lemmas. Two of them will be established in separate subsections.

\begin{lem}\label{TrPalin} For any sequences of integers $c_1,c_2,\ldots, c_{k}$ one has:
$$\Tr  M_{q}(c_1,c_2,\ldots, c_{k}) = \Tr  M_{q}(c_k,c_{k-1},\ldots, c_{1}). $$
\end{lem}

The proof of Lemma \ref{TrPalin} is postponed to \S\ref{pfPalin}. 

\begin{lem}\label{trace}
For any sequences of integers $c_1,c_2,\ldots, c_{k}$ the polynomial of $\Z[q,q^{-1}]$ given by the trace $\Tr M_{q}(c_1,c_2,\ldots, c_{k}) $ is a a palindrome. 
\end{lem}

\begin{proof}
%First we note that the trace of an element in $\PG$ is defined up to a rescaling by a power of $q$. This does not affect the property of palindrome for a polynomial.
%Every elements $M_{q}\in \PG$ can be represented as $M_{q}=M_{q}(c_1,c_2,\ldots, c_{k})$ for some 
%sequence of integers $c_1,c_2,\ldots, c_{k}$. 
Using matrix transpose one sees that
$$
M_{q}(c_k,c_{k-1},\ldots, c_{1})=q^{\sum_{i=1}^{k}(c_{i}-1)}M_{q^{-1}}(c_1,c_2,\ldots, c_{k})
$$
Combining this identity with Lemma \ref{TrPalin} one deduces the equality
$$
\Tr  M_{q}(c_1,c_2,\ldots, c_{k})=q^{\sum_{i=1}^{k}(c_{i}-1)} \Tr M_{q^{-1}}(c_1,c_2,\ldots, c_{k}),
$$
which exactly says that the polynomial $\Tr  M_{q}(c_1,c_2,\ldots, c_{k}) $ is a palindrome.
\end{proof}

\begin{lem}\label{TrPos}
For any sequences of positive integers $c_1,c_2,\ldots, c_{k}$ with $c_1,\ldots, c_{k-1}$ greater than 2,
the polynomial of $\Z[q]$ given by the trace $\Tr M_{q}(c_1,c_2,\ldots, c_{k}) $ has positive coefficients.
\end{lem}

The proof of Lemma \ref{TrPos} is postponed to \S\ref{pfPos}. \\

Let $M$ be an element of $\PSL(2,\Z)$. Since the element $[M]_{q}$ of $\PG$ is a classe modulo  $\pm q^{N}\Id$, the trace of $[M]_{q}$ is a polynomial defined up to a factor of the form $\pm q^{N}$. 
Every element of $\PSL(2,\Z)$ can be written as $M(c_{1}, \ldots, c_{k})$ with positive integer coefficients $c_{i}$'s. This is a consequence of the following decompositions of the generators, that can be checked by direct computations,
$$
R=-M(2,1,1), \quad R^{-1}=-M(1,1,2,1), \quad S=-S^{-1}=-M(1,1,2,1,1).
$$
In addition, there exists a decomposition where the coefficients are all greater than two except perhaps for the first and last ones.
\begin{lem}[{\cite[Prop. 7.4]{MGOfarey1}}]\label{lemfarey}
Every matrix $M\in \PSL(2,\Z)$ can be written as $M=M(c_{1}, \ldots, c_{k})$ with all $c_{i}\geq 2$ except perhaps for $c_{1}$ or simultaneously $c_{1}, c_{2}$, and for $c_{k}$ or simultaneously $c_{k-1}, c_{k}$ which can be equal to one.\end{lem}

The trace of the matrix $M_{q}(c_{1}, \ldots, c_{k})$ is invariant under a cyclic permutation of the coefficients $c_{i}$ and one has $M_{q}(1,1,1)=-\Id$. This implies by Lemma \ref{lemfarey} that the trace of $[M]_{q}$ is equal, up to a factor $\pm q^{N}$, to the trace of a matrix $M_{q}(c_{1}, \ldots, c_{k})$ with $c_1,\ldots, c_{k-1}$ greater than 2. Hence, Lemma \ref{TrPos} and Lemma \ref{trace} apply and this will prove Theorem \ref{tracePG}.\\

In addition we formulate the following conjecture.
\begin{conj}
For any sequences of integers $c_1,c_2,\ldots, c_{k}$ the polynomial of $\Z[q,q^{-1}]$ given by $\Tr M_{q}(c_1,c_2,\ldots, c_{k}) $ has unimodal sequence of coefficients.
\end{conj}

\subsection{Examples: Cohn matrices}
We consider the following 2 matrices of $\SL_{}(2,\Z)$:
$$
A:=\begin{pmatrix}
2&1\\
1&1
\end{pmatrix}=-M(2,2,1,1), \qquad
B:=\begin{pmatrix}
5&2\\
2&1
\end{pmatrix}=-M(3,2,2,1,1).
$$
These matrices and their products according to Christoffel words provide with all the Markov numbers, see e.g. \cite{Reu}. All these matrices are also known as Cohn matrices. Markov numbers appear twice in the matrices, as the entries in the upper right corner as well as the third of the traces of the matrices. The $q$-deformations of the matrices will lead to $q$-analogues of Markov numbers. We compute the $q$-deformations of the first Cohn matrices.
$$
\begin{array}{llc}
&[A]_{q}=
\begin{pmatrix}
q+q^{2}&1\\
q&1
\end{pmatrix}, \Tr [A]_{q}=1+q+q^{2}\\[12pt]
&[B]_{q}=
\begin{pmatrix}
q+2q^{2}+q^{3}+q^{4}&1+q\\
q+q^{2}&1
\end{pmatrix}, \Tr [B]_{q}=(1+q+q^{2})(1+q^{2})\\[12pt]
&[AB]_{q}=
\begin{pmatrix}
q+2q^{2}+3q^{3}+3q^{4}+2q^{5}+q^{6}&1+q+2q^{2}+q^{3}\\
q+2q^{2}+2q^{3}+q^{4}+q^{5}&1+q+q^{2}
\end{pmatrix},\\[10pt]
& \Tr [AB]_{q}=(1+q+q^{2})(1+q+q^{2}+q^{3}+q^{4})\\[12pt]
&[A^{2}B]_{q}=
\begin{pmatrix}
q+3q^{2}+5q^{3}+6q^{4}+7q^{5}+5q^{6}+3q^{7}+q^{8}&1+2q+3q^{2}+3q^{3}+3q^{4}+q^{5}\\
q+3q^{2}+4q^{3}+4q^{4}+4q^{5}+2q^{6}+q^{7}&1+q+2q^{2}+2q^{3}+q^{4}
\end{pmatrix},\\[10pt]
& \Tr [A^{2}B^{}]_{q}=(1+q+q^{2})(1+2q+2q^{2}+3q^{3}+3q^{4}+2q^{5}+q^{6})\\[12pt]
\end{array}
$$

$$
\begin{array}{llc}
&[A^{}B^{2}]_{q}=\\[4pt]
&\begin{pmatrix}
q+3q^{2}+7q^{3}+11q^{4}+13q^{5}+13q^{6}+11q^{7}+7q^{8}+3q^{9}+q^{10}&1+2q+5q^{2}+6q^{3}+6q^{4}+5q^{5}+3q^{6}+q^{7}\\[10pt]
q+3q^{2}+6q^{3}+8q^{4}+8q^{5}+7q^{6}+5q^{7}+2q^{8}+q^{9}&1+2q+4q^{2}+4q^{3}+3q^{4}+2q^{5}+q^{6}
\end{pmatrix},\\[14pt]
& \Tr [A^{}B^{2}]_{q}=(1+q+q^{2})(1+2q+4q^{2}+5q^{3}+5q^{4}+5q^{5}+4q^{6}+2q^{7}+q^{8})\\[12pt]
&[A^{3}B^{}]_{q}=\\[4pt]
&\begin{pmatrix}
q+4q^{2}+8q^{3}+12q^{4}+15q^{5}+15q^{6}+13q^{7}+8q^{8}+4q^{9}+q^{10}&1+3q+5q^{2}+7q^{3}+7q^{4}+6q^{5}+4q^{6}+q^{7}\\
q+4q^{2}+7q^{3}+9q^{4}+10q^{5}+9q^{6}+6q^{7}+3q^{8}+q^{9}&1+3q+4q^{2}+5q^{3}+4q^{4}+3q^{5}+q^{6}
\end{pmatrix},\\[12pt]
& \Tr [A^{3}B^{}]_{q}=(1+q+q^{2})(1+q^{2})(1+3q+3q^{2}+3q^{3}+3q^{4}+3q^{5}+q^{6})\\[12pt]
\end{array}
$$
We observe that the traces are always divisible by $[3]_{q}$. Note that the same approach to $q$-deformations of Cohn matrices based on the $q$-rationals was introduced in \cite{Kog}. Our computations coincide up to a power of $q$ due to different initial deformations of the matrices $A$ and $B$.

\subsection{Proof of Lemma \ref{TrPalin}}\label{pfPalin}
We will proceed by induction. We start with useful relations on the matrices $M_{q}(c_1,c_2,\ldots, c_{k})$ and their traces which allow some reductions on the sequence of coefficients.

\begin{lem}\label{reduc} For all integers $c_1,c_2,\ldots, c_{k}$, $c$ and $d$ one has

(i)  $M_{q}(c, 1, d)=qM_{q}(c-1, d-1)$;

(ii) $M_{q}(c, -1, d)=-q^{-2}M_{q}(c+1, d+1)$;

(iii) $\Tr  M_{q}(c_1,c_2,\ldots, c_{k},0)=-q^{-1}\Tr  M_{q}(c_1+c_{k},c_2,\ldots, c_{k-1})$.

\end{lem}

\begin{proof}
Items (i) and (ii) can be checked by direct computations on the matrices. For item (iii) we write 
$$M_{q}(c_1,c_2,\ldots, c_{k},0)=R_{q}^{c_1}S_{q}\,R_{q}^{c_{2}}S_{q}\cdots{}R_{q}^{c_k}\underbrace{S_{q}R_{q}^{0}S_{q}}_{{-q^{-1}\Id}}=-q^{-1}M_{q}(c_1,c_2,\ldots, c_{k-1})R_{q}^{c_k}.$$
Taking the trace we obtain
$$\Tr M_{q}(c_1,c_2,\ldots, c_{k},0)=-q^{-1}\Tr (M_{q}(c_1,c_2,\ldots, c_{k-1})R_{q}^{c_k})=-q^{-1}\Tr  M_{q}(c_1+c_{k},c_2,\ldots, c_{k-1}),$$
which gives (iii).
\end{proof}

\noindent
We consider the following induction hypothesis: 

\noindent
(H) $\Tr  M_{q}(c_1,c_2,\ldots, c_{k}) = \Tr  M_{q}(c_k,c_{k-1},\ldots, c_{1})$, for any sequence of integers $c_1,c_2,\ldots, c_{k}$.\\

%\noindent
%(H2) $\Tr  M_{q}(c_1,c_2,\ldots, c_{k},0) = \Tr  M_{q}(0,c_k,c_{k-1},\ldots, c_{1})$, for any sequence of integers $c_1,c_2,\ldots, c_{k}$.\\

\noindent
Property (H) is clear for $k=2$, since $M_{q}(c_{1},c_{2})=M_{q}(c_{1})M_{q}(c_{2})$.\\

We now assume that (H) holds up to some fixed $k\geq 2$.

Let us fix $(c_1,c_2,\ldots, c_{k})$ and introduce notation for the entries of the matrices
\begin{equation}\label{not}
M_{q}(c_1,c_2,\ldots, c_{k}) =
\begin{pmatrix}
A&B\\
C&D
\end{pmatrix},
\qquad 
M_{q}(c_k,c_{k-1},\ldots, c_{1})=
\begin{pmatrix}
\bar A& \bar B\\
\bar C& \bar D
\end{pmatrix},
\end{equation}

We will need the following preliminary computations:
\begin{equation}\label{t00}
\begin{array}{lclcl}
M_{q}(c_1,c_2,\ldots, c_{k}, c)&=&\begin{pmatrix}
A&B\\
C&D
\end{pmatrix}
\begin{pmatrix}
[c]&-q^{c-1}\\
1&0
\end{pmatrix}
&=&\begin{pmatrix}
[c]A+B&-q^{c-1}A\\
[c]C+D&-q^{c-1}C
\end{pmatrix}
\\[14pt]
M_{q}(c, c_k,c_{k-1},\ldots, c_{1})&=&
\begin{pmatrix}
[c]&-q^{c-1}\\
1&0
\end{pmatrix}
\begin{pmatrix}
\bar A& \bar B\\
\bar C& \bar D
\end{pmatrix}
&=&\begin{pmatrix}
[c]\bar A-q^{c-1}\bar C&[c]\bar B-q^{c-1}\bar D\\
\bar A&\bar B
\end{pmatrix}
\end{array}
\end{equation} 
where $c$ is any integer and $[c]=[c]_{q}$.

By (H) one has the following relation between the entries of the matrices \eqref{not}:
\begin{equation*}\label{t1}
A+D=\bar A+\bar D.
\end{equation*}

\noindent
\textbf{Step 1}. We show that in addition one has the following relation in the  matrices \eqref{not}:
\begin{equation}\label{t1b}
C-q^{}B=\bar C -q^{}\bar B.
\end{equation}
Using \eqref{t00} with $c=0$ we see that 
$$
C-q^{}B=-q\Tr M_{q}(c_1,c_2,\ldots, c_{k},0) \quad \text {and } \quad\bar C -q^{}\bar B=-q \Tr  M_{q}(0,c_k,c_{k-1},\ldots, c_{1}).
$$
The equality between these two is obtained using Lemma \ref{reduc} (iii) that allows to reduce the length of the cycle of coefficients and then by applying (H).\\

\noindent
\textbf{Step 2}. We show that in addition one has the following relation in the  matrices \eqref{not}:
\begin{equation}\label{t2}
A+B-C=\bar A+\bar B-\bar C
\end{equation}
Using \eqref{t00} with $c=1$ we see that 
$$
A+B-C=\Tr  M_{q}(c_1,c_2,\ldots, c_{k}, 1), \qquad \bar A+\bar B-\bar C=\Tr  M_{q}(1, c_k,c_{k-1},\ldots, c_{1}).
$$
The invariance of the trace by cyclic permutations and Lemma \ref{reduc} (i) give on the one hand 
\begin{eqnarray*}
\Tr  M_{q}(c_1,c_2,\ldots, c_{k}, 1) &=&\Tr  M_{q}(c_2,\ldots, c_{k}, 1, c_1) \\ 
&=&q\Tr  M_{q}(c_2,\ldots, c_{k}-1, c_1-1)
\end{eqnarray*}
and on the other hand 
\begin{eqnarray*}
\Tr  M_{q}(1, c_k,c_{k-1},\ldots, c_{1}) &=&\Tr M_{q}(c_{1},1, c_k,c_{k-1},\ldots, c_{2})   \\ 
&=&q\Tr M_{q}(c_{1}-1, c_k-1,c_{k-1},\ldots, c_{2})
\end{eqnarray*}
By (H) we obtain \eqref{t2}.\\

\noindent
\textbf{Step 3}. We want to show: 
\begin{eqnarray}\label{t0}
\Tr  M_{q}(c_1,c_2,\ldots, c_{k}, c_{k+1}) &=& \Tr  M_{q}(c_{k+1}, c_k,c_{k-1},\ldots, c_{1}),  %\text{ and }\\
%\label{t0b}
% \Tr  M_{q}(c_1,c_2,\ldots, c_{k}, c_{k+1},0) &=& \Tr  M_{q}(0,c_{k+1}, c_k,c_{k-1},\ldots, c_{1})
\end{eqnarray}
We proceed by induction on the integer $c_{k+1}\geq 0$. 

The cases $c_{k+1}=0$ and $c_{k+1}=1$ have been established at steps 1 and 2.
We then assume that \eqref{t0} holds for $c_{k+1}=c>0$. 
By  \eqref{t00} this gives us the relation:
\begin{equation}\label{t4}
\begin{array}{lcl}
[c]A+B-q^{c-1}C&=&[c]\bar A+\bar B -q^{c-1}\bar C.
%\\[6pt]
%q^{c-1}A+q^{-1}[c]C+q^{-1}D&=&-[c]\bar B+q^{c-1}\bar D+q^{-1}\bar A.
\end{array}
\end{equation}
Now we compute
\begin{equation*}\label{t5}
\begin{array}{lcl}
\Tr  M_{q}(c_1,c_2,\ldots, c_{k}, c+1)&=&
[c+1]A+B-q^{c}C\\[4pt]
&=&q[c]A+A+B-q^{c}C\\[4pt]
&=&(q[c]A+qB-q^{c}C)+(A+B-C)+(C-qB)\\[4pt]
&=&(q[c]\bar A+q\bar B-q^{c}\bar C)+ (\bar A+\bar B-\bar C)+(\bar C-q\bar B)\\[4pt]
&=& [c+1]\bar A+\bar B-q^{c}\bar C\\[4pt]
&=& \Tr  M_{q}(c+1, c_k,c_{k-1},\ldots, c_{1}),
\end{array}
\end{equation*}
where we have used $[c+1]=q[c]+1$ and all the relations \eqref{t4}, \eqref{t2}, \eqref{t1b}.\\

At this stage we have proved that \eqref{t0} holds for all $c_{k+1}\geq 0$.
The case with $c_{k+1}\leq 0$ will be established the same way. We will replace the relation of Step 2 with the one given by the identity
$$\Tr  M_{q}(c_1,c_2,\ldots, c_{k}, -1)=\Tr  M_{q}(-1, c_k,c_{k-1},\ldots, c_{1})$$
%$$\Tr  M_{q}(c_1,c_2,\ldots, c_{k}, -1,0)=\Tr  M_{q}(0,-1, c_k,c_{k-1},\ldots, c_{1})$$
that can be obtained with Lemma \ref{reduc} (ii). And then we will proceed at step 3 by a decreasing induction on the integer $c_{k+1}<0$ using $[c-1]=q^{-1}[c]-q^{-1}$.

%%%%%%%%%%%%%%%%%%%%%%%%%%%%%
\subsection{Proof of Lemma \ref{TrPos}}\label{pfPos}
%%%%%%%%%%%%%%%%%%%%%%%%%%%%%%

We will mimic the proof of Lemma \ref{TrPalin}.
The induction hypothesis is now

(H') $\Tr  M_{q}(c_1,c_2,\ldots, c_{k})$ is a polynomial with positive integer coefficients for any sequence of positive integers $c_1,c_2,\ldots, c_{k}$ with $c_{i}\geq 2$, $\forall i<k$.

The property (H') can be checked by direct computations for $k=1$ and $k=2$.

We proceed as in the proof of Lemma \ref{TrPalin}.
At step 1 we will obtain
$$
%A+D\; \in \Z_{\geq0}[q], \qquad 
C-qB\; \in \Z_{\geq0}[q].
$$
At step 2 we will obtain
$$
A+B-C \; \in \Z_{\geq0}[q].
% \qquad A+q^{-1}C+q^{-1}D\in \Z_{\leq0}[q,q^{-1}],
$$
At step 3 assuming
$$\Tr  M_{q}(c_1,c_2,\ldots, c_{k}, c)=[c]A+B-q^{c-1}C  \; \in \Z_{\geq0}[q],
$$
we will be able to deduce wit the same computation
$$
\Tr  M_{q}(c_1,c_2,\ldots, c_{k}, c+1)=[c+1]A+B-q^{c}C\; \in \Z_{\geq0}[q].
$$
This established (H') by induction.

\begin{rem}
In the proves of Lemmas \ref{TrPalin} and \ref{TrPos} we have established properties on the entries of the matrices $M_{q}(c_1,c_2,\ldots, c_{k})$ that translate as follows in terms of $q$-continuants
\begin{eqnarray*}
E_{k}(c_1,c_2,\ldots, c_{k})_{q}-q^{c_{k}-1}E_{k-2}(c_2,\ldots, c_{k-1})_{q}&=&
E_{k}(c_k,c_{k-1},\ldots, c_{1})_{q}-q^{c_{1}-1}E_{k-2}(c_{k-1},\ldots, c_{2})_{q}\\[6pt]
E_{k-1}(c_2,\ldots, c_{k})_{q}-q^{c_{k}}E_{k-1}(c_1,\ldots, c_{k-1})_{q}&=&
E_{k-1}(c_{k-1},\ldots, c_{1})_{q}-q^{c_{1}}E_{k-1}(c_{k},\ldots, c_{2})_{q}
\end{eqnarray*}
Moreover, these polynomials are palindromes and when $c_1,c_2,\ldots, c_{k}$ are positive integers they have positive integer coefficients.
\end{rem}

%\subsection{Relations on the continuants}
%%%%%%%%%%%%%%%%%      
\section{Quadratic irrationals}\label{secquad}
%%%%%%%%%%%%%%%%%%%
\subsection{Real quadratic irrational numbers}
A  real quadratic irrational number is a real number $x$ of the form $x=\frac{r\pm \sqrt{p}}{s}$, with $r\in\Z$, $s, p\in\Z_{>0}$ and $p$ is not a square in $\Z_{>0}$. The following assertions are equivalent:

(a) $x$ is a real quadratic irrational number,

(b) $x$ is a solution of an equation $aX^{2}+bX+c=0$, with $a,b,c\in \Z$ and $b^{2}-4ac$ positive and not a square,

(c) there exists $M\in \SL(2, \Z)$, with $\Tr M>2$, such that $M\cdot x=x$,

(d) $x$ has a periodic infinite continued fraction expansion [Lagrange Theorem].

Theorem \ref{result1} stated in the Introduction says that the $q$-deformations that we have introduced behaves nicely with the characterization of quadratic irrationals. 

\subsection{Proof of Theorem \ref{result1}}
Item (iii) of Theorem \ref{result1} is now a corollary of Proposition \ref{action} and this immediately implies Item (i) (except for the palindromicity property of $\Pc$ that will be established in \S\ref{palP}) and  Item (ii). Item (iv) is a tautology coming from \S\ref{infCF} and Lagrange theorem (the property (d) of quadratic irrational recalled in the previous subsection).

\subsection{Explicit expressions}

As recalled in the first paragraph $x$ is a quadratic irrational number if and only if its expansion as a continued fraction becomes periodic, i.e. if and only if there exist integers $(b_1, \ldots, b_l)$, with $b_{i}\geq 2$ for $i\geq 2$, and integers $(c_1,\ldots,c_k)\not=(2, \ldots, 2)$, with $c_{i}\geq 2$ for all $i$,  such that
$$
x=\llbracket{}b_1, \ldots, b_l, {c_1,\ldots,c_k}, {c_1,\ldots,c_k}, {c_1,\ldots,c_k},\dots \rrbracket{}
$$
One writes
$x=\llbracket{}b_1, \ldots, b_l, \overline{c_1,\ldots,c_k}\rrbracket{} $. When 
$x=\llbracket{} \overline{c_1,\ldots,c_k}\rrbracket{} $ one says that $x$ has a purely periodic expansion.

\begin{lem}
If  $x=\llbracket{} \overline{c_1,\ldots,c_k}\rrbracket{} $ then $[x]_{q}$ is a fixed point of the matrix $M_{q}(c_1,\ldots,c_k)$.
\end{lem}

\begin{proof}
It is a classical result that if $x=\llbracket{} \overline{c_1,\ldots,c_k}\rrbracket{} $ then $x$ is a fixed point of the matrix $M(c_1,\ldots,c_k)$. Applying Proposition \ref{action} one gets the $q$-analogue of this result.
\end{proof}

\begin{prop}\label{ABC}
If $x$ is a fixed point of the matrix $M(c_1,\ldots,c_k)$ then its $q$-deformation $[x]_{q}$ satisfies the equation
$$
\A X^{2}-\B X+\cC=0,
$$
where $\A$, $\B$ and $\C$ are polynomials in $\Z[q,q^{-1}]$ given by $q$-continuants
\begin{eqnarray*}
\A&=&E_{k-1}({c_2,\ldots,c_k})_{q} \\
\B&=& E_{k}({c_1,\ldots,c_k})_{q}+q^{c_{k}-1}E_{k-2}({c_2,\ldots,c_{k-1}})_{q} \\
\cC&=&q^{c_{k}-1}E_{k-1}({c_1,\ldots,c_{k-1}})_{q} 
\end{eqnarray*}

\end{prop}

\begin{proof}
Expressing that $[x]_{q}$ is a fixed point of $M_{q}(c_1,\ldots,c_k)$ and using the description \eqref{q-cont} one gets
$$
[x]_{q}= \frac{[x]_{q}E_{k}(c_{1}, c_{2},\ldots, c_{k})_{q}-q^{c_{k}-1}E_{k-1}(c_{1}, c_{2},\ldots, c_{k-1})_{q}}{[x]_{q}E_{k-1}(c_{2}, c_{3},\ldots, c_{k})_{q}-q^{c_{k}-1}E_{k-2}(c_{2}, c_{3},\ldots, c_{k-1})_{q}},
$$
which leads to the result.
\end{proof}

 \begin{prop}\label{PRS}
If $x$ is a fixed point of the matrix $M(c_1,\ldots,c_k)$, then its $q$-deformation $[x]_{q}$  has the following form
\begin{equation*}
[x]_{q}=\dfrac{\Rc\pm\sqrt\Pc}{\Sc},
\end{equation*}
where $\Pc$, $\Rc$ and $\Sc$ are polynomials in $\Z[q, q^{-1}]$ given by $q$-continuants
\begin{eqnarray*}
\Pc&=& \Pc=(\Tr M_{q}({c_1,\ldots,c_k}))^{2}-4q^{\sum_{i=1}^{k}(c_{i}-1)}\\
\Rc&=& E_{k}({c_1,\ldots,c_k})_{q}+q^{c_{k}-1}E_{k-2}({c_2,\ldots,c_{k-1}})_{q}\\
\Sc&=&2E_{k-1}({c_2,\ldots,c_k})_{q}
\end{eqnarray*}
\end{prop}

\begin{proof}
The expressions are obtained by solving the equation of Proposition \ref{ABC}. 
%We choose the solution with the positive sign because we want that $q=1$ returns $x$.
The expression of $\Pc$ may be simplified using 
$$\Tr M_{q}({c_1,\ldots,c_k})=E_{k}({c_1,\ldots,c_k})_{q}-q^{c_{k}-1}E_{k-2}({c_2,\ldots,c_{k-1}})_{q}$$ and using the following relation on the $q$-continuants
$$
E_{k-1}({c_2,\ldots,c_k})_{q} \cdot q^{c_{k}-1}E_{k-1}({c_1,\ldots,c_{k-1}})_{q} 
-E_{k}({c_1,\ldots,c_k})_{q}\cdot q^{c_{k}-1}E_{k-2}({c_2,\ldots,c_{k-1}})_{q}
=q^{\sum_{i=1}^{k}(c_{i}-1)}.
$$
which is given by Desnanot-Jacobi identity in \eqref{KEq}.
\end{proof}

\subsection{Palindromicity of $\Pc$ (end of proof of Theorem \ref{result1}) }\label{palP}
Let $x$ be a quadratic irrational. There exists a matrix $M\in \PSL(2, \Z)$ such that $M\cdot x =x$. 
One can write $M=M({c_1,\ldots,c_k})$ with $c_1,\ldots,c_k$ positive integers. We obtain $[x]_{q}$  as the fixed point of $M_{q}({c_1,\ldots,c_k})$. Since the coefficients $c_{i}$'s are positive, the entries of $M_{q}({c_1,\ldots,c_k})$ are polynomials in $q$ and by Proposition \ref{PRS} we obtain $[x]_{q}=\frac{\Rc\pm\sqrt\Pc}{\Sc},
$ with $$\Pc=(\Tr M_{q}({c_1,\ldots,c_k}))^{2}-4q^{\sum_{i=1}^{k}(c_{i}-1)}.$$
We know by Lemma \ref{trace} that the polynomial $\Tr M_{q}({c_1,\ldots,c_k})$ is a palindrome. Moreover its degree is the same as $\deg E_{k}(c_1,\ldots,c_k)_{q}=\sum_{i=1}^{k}(c_{i}-1)$.
Hence, $(\Tr M_{q}({c_1,\ldots,c_k}))^{2}$ is a palindromic polynomial of even degree  with median coefficient attached to $q^{\sum_{i=1}^{k}(c_{i}-1)}$. Thus $\Pc$ is still a palindrome.

\subsection{Examples}\label{exquad} We study examples of quadratic irrationals with short periodic continued fraction expansions.

\begin{ex}
Let us consider the quadratic irrational with purely periodic expansion with sign $-$ of period 1:
$$x=\llbracket \bar c\rrbracket_{}=c - \cfrac{1}{c 
          - \cfrac{1}{c_{} -
           \cfrac{1}{\ddots} }}=\dfrac{c+\sqrt{c^{2}-4}}{2}, \qquad c\geq 3.$$
The $q$-deformation leads to
$$
[x]_{q}=\dfrac{[c]_{q}+\sqrt{[c]_{q}^{2}-4q^{c-1}}}{2}.
$$
For $c=3$ one obtains
$$
\left[\frac{3+\sqrt5}{2}\right]_{q}=\frac{1+q+q^{2}+\sqrt{1+2q-q^{2}+2q^{3}+q^{4}}}{2}=\frac{1+q+q^{2}+\sqrt{(1-q+q^{2})(1+3q^{}+q^{2})}}{2}
$$
Starting from $c\geq 4$ the polynomial $[c]_{q}^{2}-4q^{c-1}$ under the radical has positive coefficients.
\end{ex}

\begin{ex}
Let us consider the quadratic irrational with purely periodic expansion with sign $+$ of period 1:
$$
x=[\overline{a}]= a+\cfrac{1}{a+{\cfrac{1}{a+\cfrac{1}{\ddots}}}}=\cfrac{a+\sqrt{a^2+4}}{2}, \qquad a\geq 1.$$
Writing $[x]_{q}=[a,\overline{a}]_{q}=[a]_{q}+\cfrac{q^a}{[a]_{q^-1}+\cfrac{q^{-a}}{[x]_{q}}}$ one gets
\begin{eqnarray*}
[x]_{q}&=&\dfrac{q[a]_{q}+(q^{a}+1)(q-1)+\sqrt{\big(q[a]_{q}+(q^{a}+1)(q-1)\big)^{2}+4q}}{2q}\\[6pt]
&=&\dfrac{q[a]_{q}+(q^{a}+1)(q-1)+\sqrt{(1-q+q^{2})([a+1]_{q}^{2}-q[2a-1]_{q}+2q^{a})}}{2q}.
\end{eqnarray*}
For $a=1,2,3,4$, we obtain respectively
\begin{eqnarray*}
\left[\frac{1+\sqrt5}{2}\right]_{q}&=& \cfrac {q^{2}+q-1+ \sqrt{(1-q+q^{2})(1+3q^{}+q^{2})}}{2q}\\[6pt]
[1+\sqrt{2}]_{q} &=& \cfrac {q^{3}+2q-1+ \sqrt{(1-q+q^{2})(1+q+4q^{2}+q^{3}+q^{4})}}{2q}\\[6pt]
\left[\cfrac{3+\sqrt{13}}{2}\right]_{q} &=& \cfrac {q^{4}+q^{2}+2q-1+ \sqrt{(1-q+q^{2})(1+q+2q^{2}+5q^{3}+2q^{4}+q^{5}+q^{6})}}{2q}\\[6pt]
[2+\sqrt{5}]_{q} &=& \cfrac {q^{5}+q^{3}+q^{2}+2q-1+ \sqrt{(1-q+q^{2})(1+q+2q^{2}+3q^{3}+6q^{4}+3q^{5}+2q^{6}+q^{7}+q^{8})}}{2q}.
\end{eqnarray*}
The polynomial under the radical factors out as follows
$$
(1-q+q^{2})(q^{2a} + q^{2a-1} + 2q^{2a-2} + \ldots + (a-1)q^{a+1} + (a+2)q^{a} + (a-1)q^{a-1} + \ldots + 2q^{2} + q + 1)$$
and starting from $a\geq 4$ it has positive coefficients.
\end{ex}

\begin{ex}
Let us consider the quadratic irrational with purely periodic expansion with sign $+$ of period 2:
$$x=[\overline{a,b}]=a+\cfrac{1}{b+{\cfrac{1}{a+\cfrac{1}{\ddots}}}}=\cfrac{ab+\sqrt{(ab)^{2}+4ab}}{2b}, \qquad a,b\geq 1.$$ 
Writing $[x]_{q}=[a,\overline{b}]_{q}=[a]_{q}+\cfrac{q^a}{[b]_{q^-1}+\cfrac{q^{-b}}{[x]_{q}}}$ one gets
\begin{eqnarray*}
[x]_{q}&=&\cfrac{q[a]_{q}[b]_{b} + q^{a+b} -1+\sqrt{\big(q[a]_{q}[b]_{b} + q^{a+b} -1\big)^{2}+4q[a]_{q}[b]_{q}}}{2q[b]_{q}}
\end{eqnarray*}
We list below the polynomials under the radical obtained for various values of $a$ and $b$:

$\bullet$ \ $a=1, b=2 \left( x=\cfrac{1+\sqrt{3}}{2} \right)$  : $q^6+2q^5+3q^4+3q^2+2q+1$ \\

$\bullet$ \ $a=1,b=3  \left( x=\cfrac{3+\sqrt{21}}{6} \right)$ : $q^8+2q^7+3q^6+4q^5+q^4+3q^2+2q+1 = (q^4+q^3+3q^2+q+1)(q^4+q^3-q^2+q+1)$ \\

$\bullet$ \ $a=1,b=4  \left( x=\cfrac{1+\sqrt{2}}{2} \right)$ : $q^{10}+2q^9+3q^8+4q^7+5q^6+2q^5+5q^4++4q^3+3q^2+2q+1$ \\

$\bullet$ \ $a=1,b=5  \left( x=\cfrac{5+3\sqrt{5}}{10} \right) $ : $q^{12}+2q^{11}+3q^{10}+4q^9++5q^8+6q^7+3q^6+6q^5+5q^4+4q^3+3q^2+2q+1 = (q^6+q^5+q^4+3q^3+q^2+q+1)(q^6+q^5+q^4-q^3+q^2+q+1)$ \\

$\bullet$ \ $a=2 , b=1  \left( x=1+\sqrt{3} \right)$ : $ q^6 + 2q^5 + 3q^4 + 3q^2 + 2q + 1$ \\

$\bullet$ \ $a=2 , b=3  \left( x=\cfrac{3+\sqrt{15}}{3} \right)$ : $ q^{10} + 2q^9 + 5q^8 + 8q^7 + 10q^6 + 8q^5 + 10q^4 + 8q^3 + 5q^2 + 2q + 1$ \\

$\bullet$ \ $a=2 , b=4  \left( x=\cfrac{2+\sqrt{6}}{2} \right)$ : $ q^{10} + 4q^8 + 8q^6 - 2q^5 + 8q^4 + 4q^2 + 1=(q^4 - q^3 + 3q^2 - q + 1)(q^6 + q^5 + 2q^4 + 2q^2 + q + 1)$ \\

$\bullet$ \ $a=2 , b=5  \left( x=\cfrac{5+\sqrt{35}}{5} \right)$ : $ q^{14} + 2q^{13} + 5q^{12} + 8q^{11} + 12q^{10} + 16q^9 + 18q^8 + 16q^7 + 18q^6 + 16q^5 + 12q^4 + 8q^3 + 5q^2 + 2q + 1$\\

$\bullet$ \ $a=3 , b=1  \left( x=\cfrac{3+\sqrt{21}}{2} \right)$ : $ q^8 + 2q^7 + 3q^6 + 4q^5 + q^4 + 4q^3 + 3q^2 + 2q + 1 = (q^4 + q^3 - q^2 + q + 1)(q^4 + q^3 + 3q^2 + q + 1)$ \\

$\bullet$ \ $a=3 , b=2  \left( x=\cfrac{3+\sqrt{15}}{2} \right)$ : $q^{10} + 2q^9 + 5q^8 + 8q^7 + 10q^6 + 8q^5 + 10q^4 + 8q^3 + 5q^2 + 2q + 1 $ \\

$\bullet$ \ $a=3 , b=4  \left( x=\cfrac{3+2\sqrt{3}}{2} \right)$ : $q^{14} + 2q^{13} + 5q^{12} + 10q^{11} + 16q^{10} + 22q^9 + 27q^8 + 26q^7 + 27q^6 + 22q^5 + 16q^4 + 10q^3 + 5q^2 + 2q + 1$ \\

$\bullet$ \ $a=3 , b=5  \left( x=\cfrac{15+\sqrt{285}}{10} \right)$ : $q^{16} + 2q^{15} + 5q^{14} + 10q^{13} + 16q^{12} + 24q^{11} + 31q^{10} + 36q^9 + 35q^8 + 36q^7 + 31q^6 + 24q^5 + 16q^4 + 10q^3 + 5q^2 + 2q + 1= (q^8 + q^7 + 2q^6 + 3q^5 + q^4 + 3q^3 + 2q^2 + q + 1)(q^8 + q^7 + 2q^6 + 3q^5 + 5q^4 + 3q^3 + 2q^2 + q + 1)$ \\
\end{ex}

\begin{ex} We list below the polynomials under the radical obtained in the $q$-deformation of $\sqrt{n}$ for the first values of 
$n$:

$ [\sqrt{2}]_{q} : q^6 + 4q^4 - 2q^3 + 4q^2 + 1=(q^2 - q + 1)(q^4 + q^3 + 4q^2 + q + 1)$ \\

$ [\sqrt{3}]_{q} : q^6 + 2q^5 + 3q^4 + 3q^2 + 2q + 1 $ \\

$ [\sqrt{5}]_{q} : q^{10} + 2q^8 + 2q^7 + 5q^6 + 5q^4 + 2q^3 + 2q^2 + 1=  (q^2 - q + 1)(q^8 + q^7 + 2q^6 + 3q^5 + 6q^4 + 3q^3 + 2q^2 + q + 1) $ \\

$ [\sqrt{6}]_{q} :q^{10} + 4q^8 + 8q^6 - 2q^5 + 8q^4 + 4q^2 + 1= (q^4 - q^3 + 3q^2 - q + 1)(q^6 + q^5 + 2q^4 + 2q^2 + q + 1) $ \\

$ [\sqrt{7}]_{q} : q^{10} + 2q^9 + q^8 + 4q^7 + 6q^6 + 6q^4 + 4q^3 + q^2 + 2q + 1, $ \\

$ [\sqrt{8}]_{q} : q^{10} + 2q^9 + 3q^8 + 4q^7 + 5q^6 + 2q^5 + 5q^4 + 4q^3 + 3q^2 + 2q + 1 $ \\

$ [\sqrt{10}]_{q} : q^{14} + 2q^{12} + 2q^{11} + 3q^{10} + 4q^9 + 7q^8 + *q^7 + 7q^6 + 4q^5 + 3q^4 + 2q^3 + 2q^2 + 1=(q^2 - q + 1)(q^{12} + q^{11} + 2q^{10} + 3q^9 + 4q^8 + 5q^7 + 8q^6 + 5q^5 + 4q^4 + 3q^3 + 2q^2 + q + 1) $ \\

$ [\sqrt{11}]_{q} : q^{14} + 2q^{12} + 4q^{11} + q^{10} + 6q^9 + 8q^8 + 8q^6 + 6q^5 + q^4 + 4q^3 + 2q^2 + 1 $ \\
\end{ex}

On all examples we compute we observe that the polynomial $\Pc$ of Theorem \ref{genexp} always factors out with a polynomial which has positive coefficients.\\

\textbf{Acknowledgement:} We are deeply grateful to Valentin Ovsienko for many valuable comments since the beginning of the project. We are also grateful to  Alexander Veselov for enlightening discussions. Significant progress on this work was made during a research in pairs stay at MFO. We thank MFO for the warm hospitality and for providing us with outstanding working conditions.
\bibliographystyle{alpha}
\bibliography{BiblioMoy3,qAnalog}
\end{document}